\documentclass[12pt]{article}
\usepackage{amscd,amsfonts,amsmath,amssymb,amstext,amsthm}

\title{Cycle space constructions for exhaustions of flag domains}

\author{Alan Huckleberry \& Joseph A. Wolf}

\date{file exhaust.tex dated July 12, 2008}
\theoremstyle{plain}
\newtheorem{theorem} {Theorem} [section]
\newtheorem{lemma} [theorem]{Lemma}
\newtheorem{proposition}[theorem]{Proposition}

\theoremstyle{definition}

\begin{document}

\maketitle

\section {Introduction}\label{sec1}

In the study of complex flag manifolds, flag domains and their cycle spaces,
a key point is the fact that the cycle space $\mathcal M_D$ of a flag domain 
$D$ is a Stein
manifold.  That fact has a long history; see \cite{FHW}.  The earliest 
approach (\cite{WeW}, \cite{W8}) relied on construction of a strictly
plurisubharmonic exhaustion function on $\mathcal M_D$, starting with
a $q$--convex exhaustion function on $D$, where $q$ is the dimension of a
particular maximal compact subvariety of $D$ (we use the normalization that
$0$--convex means Stein).  Construction of that exhaustion function on $D$
\cite{SW} required that $D$ be measurable \cite{W2}.  In that case the 
exhaustion on $D$ was transferred to $\mathcal M_D$, using a special
case of a method due to
Barlet \cite{B}.  Here we do the opposite: we use the incidence method of
\cite{FHW} to construct a canonical plurisubharmonic exhaustion function on
$\mathcal M_D$ and use it in turn to construct a canonical $q$--convex 
exhaustion function on $D$.  This promises to have strong
consequences for cohomology vanishing theorems and the construction of
admissible representations of real reductive Lie groups.  It also
completes the argument of \cite{W9}.
\medskip

We remark that our construction only produces a continuous exhaustion
of the cycle space $D$.  However, that exhaustion is $q$--pseudoconvex
in the strong sense that it is optimally locally approximated by a
smooth $q$--convex function.  (See Section \ref{sec5} for details.)
Assuming that the Andreotti--Grauert theory of pseudoconvexity and
cohomology vanishing \cite{AnG} remains valid for such exhaustions
there will be direct applications to the geometric construction and 
analysis of admissible representations of real reductive Lie groups.
\medskip

In this context we want to draw the reader's attention to the construction
of Kr\" otz and Stanton (\cite{KS1}, \cite{KS2}) for plurisubharmonic 
exhaustion functions on $\mathcal M_D$ associated to spherical 
representations and their positive definite spherical functions.
\medskip

Now we establish the basic notation of this paper.
Let $G$ be a connected, simply connected complex semisimple Lie group,
$\mathfrak g$ its Lie algebra, $\mathfrak g_0$ a real form of
$\mathfrak g$, and $G_0$ the corresponding real form of $G$.  Fix a
parabolic subgroup $Q \subset G$ and let $Z$ denote the complex flag manifold
$G/Q$.  Then $G_0$ has only finitely many orbits on Z, in particular has
open orbits \cite{W2}.  By {\em flag domain} we mean an open orbit 
$D = G_0(z_0)$.
\medskip

Fix a maximal compact subgroup $K_0 \subset G_0$, let $K \subset G$ denote
its complexification, and let $\mathfrak k_0$ and $\mathfrak k$ denote
their respective Lie algebras.  There is a unique $K_0$--orbit (call it
$C_0$) in the flag domain $D$ that is a complex submanifold of $D$.  
Replacing $K_0$ by a $G_0$--conjugate if necessary we may assume that
this {\em base cycle} $C_0 = K_0(z_0)$, and then $C_0 = K(z_0)$.
The {\em cycle space} $\mathcal M_D$ is the topological component of
$C_0$ in $\{gC_0 \mid g \in G \text{ and } gC_0 \subset D\}$.
\medskip

Let $\theta$ denote the Cartan involution of $G_0$ with fixed point
set $K_0$, and decompose $\mathfrak g_0 = \mathfrak k_0 + \mathfrak s_0$ 
into $(\pm1)$--eigenspaces of $d\theta$.  Then $\mathfrak g_u :=
\mathfrak k_0 +\sqrt{-1}\,\mathfrak s_0$ is a real subalgebra, in fact a
real form, of $\mathfrak g$.  The corresponding analytic subgroup
$G_u \subset G$ is a compact real form of $G$.
\medskip

We view the base cycle $C_0$ as a point in the complex space 
$\mathcal C_q(Z)$ of $q$--dimensional cycles in the projective algebraic
variety $Z$. There the group $G$ acts
algebraically as a group of holomorphic transformations, and
consequently the orbit $\mathcal M_Z:=\{gC_0\mid g\in G\}$
is a complex $G$--homogeneous space.  Since $D$ is open in $Z$,
the subspace $\mathcal M_Z\cap \mathcal C_q(D)$ (cycles that
are contained in $D$) is open in $\mathcal M_Z$, and 
$\mathcal M_D$ is its connected component containing $C_0$.
In other words, $\mathcal M_D$ is the complex submanifold
$\{gC_0 \mid g \in G \text{ and } gC_0 \subset D\}$ of $\mathcal C_q(D)$
as well as an open submanifold of $\mathcal M_Z$.  In fact $\mathcal M_D$ 
is closed in $\mathcal C_q(D)$ [HH].
As mentioned above, it is a Stein manifold.
\medskip

The flag domain $D$ and its cycle space $\mathcal M_D$ fit into a {\em 
holomorphic double fibration} as follows.  The {\em incidence space}
$\mathfrak X_D := \{(z,C) \in D\times \mathcal M_D \mid z \in C\}$
is a complex submanifold of $D\times \mathcal M_D$ and the projections
are holomorphic fibrations:
$$
\setlength{\unitlength}{.08 cm}
\begin{picture}(150,18)
\put(81,15){$\mathfrak X_D$}
\put(70,12){$\mu$}
\put(66,1){$D$}
\put(90,12){$\nu$}
\put(95,1){$\mathcal M_D$}
\put(78,13){\vector(-1,-1){6}}
\put(87,13){\vector(1,-1){6}}
\end{picture}
$$
Note that $\mu$ is a holomorphic submersion and $\nu$ is a proper holomorphic
map which is a locally trivial bundle.

\section {Incidence geometry and Schubert fibrations}
\label {incidence geometry} \label{sec2}

Here we summarize several methods and results from our previous work which 
will be of use in the present note. These were
systematically presented in \cite {FHW}.  The setting, as just described,
is that of an open $G_0$--orbit $D$ in a complex flag manifold
$Z=G/Q$.  The maximal compact subgroup $K_0$ and the base cycle $C_0$
are constructed as above, and $q$ denotes the complex dimension of $C_0$.
\medskip

Given an Iwasawa decomposition $G_0=K_0A_0N_0$, we write $A$ and $N$ for the
complexifications of $A_0$ and $N_0$, and we consider Borel subgroups 
$B$ of $G$ that contain the solvable subgroup $AN$.  Those are the
{\em Iwasawa Borel} subgroups.  Any two Iwasawa decompositions 
$G_0=K_0A_0N_0$ and $G_0=K_0A'_0N'_0$ are $K_0$--conjugate, and it follows
that any two Iwasawa--Borel subgroups (for the same choice of $K_0$) are
$K_0$--conjugate.
\medskip

If $B \subset G$ is any Borel subgroup and $z\in Z$,
then as usual we refer to the closure $S := \mathrm {c}\ell (B(z))$
as the \emph{Schubert variety} associated to $B$ and $z$.
The orbit $\mathcal O_S:=B(z)$, which is Zariski open in
$S$, is the associated \emph{Schubert cell}.  The \emph{incidence
variety} $H_S\subset \mathcal M_Z$ consists of those cycles
$gC_0$, $g\in G$, which have nonempty intersection with the
boundary $B_S:=S\setminus \mathcal O_S$ of the Schubert cell
in the Schubert variety.  Below we will consider special
points $z\in C_0$ where, when $B$ an Iwasawa--Borel subgroup relative
to a decomposition $G_0=K_0A_0N_0$, the orbit $\Sigma :=A_0N_0(z)$ is 
exactly the connected component of $S\cap D$ that contains $z$.  In that
setting we refer to $\Sigma $ as a \emph{Schubert slice}.
\medskip

Since $C_0$ defines a nonzero class in $H_{2q}(Z,\mathbb Z)$, it has 
nonempty intersection with at least one $q$--codimensional $B$--Schubert 
variety $S$.  Under the additional assumption that $B\supset AN$, such
Schubert varieties have the following properties 
(see \cite {FHW}, Chapter II.7):
\begin {itemize}
\item
The intersection $C_0 \cap S$ consists of finitely many
points $z_1,\ldots ,z_m$.
\item
These points are contained in the intersection 
$D\cap \mathcal O_S$ of $D$ with the Schubert cell (open $B$--orbit in $S$).
\item
For $1 \leqq j \leqq m$ the connected component $\Sigma _j$
of $S\cap D=\mathcal O_S\cap D$ which contains $z_j$
is the Schubert slice $A_0N_0(z_j)$.
\item
Every cycle $C\in \mathcal M_D$ intersects each Schubert slice  $\Sigma _j$ in
exactly one point.
\item
If $\Sigma $ is a Schubert slice then the incidence
map $\pi _{_\Sigma} : \mathcal M_D\to \Sigma $, which associates to 
$C$ its point of intersection with $\Sigma $, is an $A_0N_0$--equivariant 
holomorphic map that is a $C^\infty$-fiber bundle. 
\item
The incidence variety $\mathcal H_S$ is a ($B$--invariant) complex 
algebraic hypersurface in $\mathcal M_Z\setminus \mathcal M_D$.  
\end {itemize}

If $C\in \mathcal M_D$ and $C\cap \Sigma = \{z\}$ 
then both the $\mu $--fiber $\mu ^{-1}(z)$ of the double fibration
and the $\pi _{_\Sigma}$--fiber through $C$ defined just above
can be identified with the set of all cycles which contain $z$. 
\medskip

Starting with $S$ and $H_S$ as above, we consider
the union of all hypersurfaces $k(H_S)$ as $k$ runs over
the compact group $K_0$.  It is a basic fact that every
boundary point of $\mathcal M_D$ in $\mathcal M_Z$ is
contained in some such $k(H_S)$.  This is formulated in
\cite {FHW} as follows. The union of the $k(H_S)$ as
$k$ runs over $K_0$ is a closed subset of $\mathcal M_Z$
and the connected component of its complement containing
$\mathcal M_D$ is a $G_0$--invariant, Kobayashi hyperbolic,
Stein domain $\mathcal E(H_S)$ (see \cite[Chapter 11]{FHW}).
The desired result is \cite[Theorem 11.3.1]{FHW}; it 
states that $\mathcal M_D$ is the maximal $G_0$-invariant,
Kobayashi hyperbolic, Stein domain in $\mathcal M_Z$ which contains
the base cycle $C_0$. This implies $\mathcal M_D=\mathcal E(H_S)$.
\medskip

For our purposes it is useful to formulate the above fact
as follows. For every cycle $C$ in the boundary of $\mathcal M_D$
in $\mathcal M_Z$ there exist a Schubert slice $\Sigma $ and an 
element $k \in K_0$
so that $C$ is both in the boundary of $k(\Sigma )$ and
the boundary $\mathcal B_{k(S)}$ of $\mathcal O_{k(S)}$ in $k(S)$.
\medskip

Of course it is possible for a sequence $\{C_n\}$ to
be divergent in $\mathcal M_D$ without converging to
a boundary point $C\in \mathcal M_Z$. But in this case,
after taking a subsequence, we may regard it as converging
to a point $C$ in the ``wonderful compactification''
of $\mathcal M_Z$.  Even in that case one observes that
$C$ is in $\mathcal B_{k(S)}$, i.e. is in the boundary of 
some $k(\Sigma )$ (\cite {HH}).

\section {Exhaustions of Schubert cells}\label{sec3}

As above $Z$ denotes the complex flag manifold $G/Q$ where $G$
is a connected, simply connected, complex semisimple group and
$Q$ is a parabolic subgroup. Thus every holomorphic line bundle
$\mathbb L\to Z$ is canonically a $G$--bundle.  We now assume that
$\mathbb L$ is positive.  Then the action of $G$ on the space
$\Gamma(Z,\mathbb L)$ of holomorphic sections is an irreducible
finite dimensional representation, say with highest weight $\lambda$.

\medskip
Since $\mathbb L$ is a line bundle, the isotropy subgroups of $G_u$ are
irreducible on its fibers, so $\mathbb L\to Z$ has a unique (up to 
multiplication by a positive real constant) $G_u$--invariant hermitian
metric. We write $\Vert \cdot \Vert$ for the
associated norm function on sections of $\mathbb L\to Z$.
\medskip

Given a section $s\in \Gamma (Z,\mathbb L)$ the
function $\eta _s:=-\mathrm{log} \Vert s\Vert ^2$ 
is a strictly plurisubharmonic exhaustion of the complement
of the zero set $\{s=0\}$ in $Z$. 
\medskip

Now let $B$ be a Borel subgroup of $G$ and 
$S=\mathcal O_S\dot \cup B_S$ a $B$--Schubert variety.
Let $\iota _{_S}: S \to Z$ be the canonical embedding and
$V_S:=\iota ^*(\Gamma (Z,\mathbb L))$ the space of restricted (to $S$)
sections. Since $V_S$ is the representation space for a rational
representation of $B$, we
have $B$--eigenvectors $s\in V_s$.  If $U$ is the unipotent
radical of $B$, then any two such eigenvectors 
are $U$-invariant and their ratio is a $U$-invariant
meromorphic function on $S$. Since $U$ has an open orbit
in $S$, that function is constant, so the two eigenvectors 
are linearly dependent.

\begin {proposition}
Let $\mathbb L\to Z$ be a positive holomorphic line bundle, $B\subset G$
a Borel subgroup, and $S \subset Z$ a $B$--Schubert variety.
Then any nonzero $B$--eigenvector $s\in V_S$
defines a strictly plurisubharmonic exhaustion 
$r_{_S}:=\eta _s\vert_{\mathcal O_S}$ of the open $B$--orbit in $S$
\end {proposition}

\begin {proof}
It remains only to show that $s$ vanishes exactly on $B_S$.
For this see \cite[Corollary 7.4.9]{FHW}.
\end {proof}

It should be underlined that having fixed $\mathbb L$ and the maximal
compact subgroup $G_u$, the exhaustion function $r_{_S}$ is 
unique.

\section {The minimax construction}\label{sec4}

Here we first construct continuous plurisubharmonic exhaustions 
$r_{_{\mathcal M_D}} $ of $\mathcal M_D$ by taking the 
supremum of lifted exhaustions $r_{_S} \circ \pi _{_\Sigma} $ 
defined by natural families of Schubert fibrations 
$\pi _{_\Sigma} :\mathcal M_D\to \Sigma \hookrightarrow S$.
We then lift $r_{_{\mathcal M_D}}$ to the universal family
$\mathfrak X_D$ and push this function down to a continuous $q$--convex
exhaustion $r_{_D}$ of $D$ by taking the infimum over the
fibers of $\mu :\mathfrak X_D\to D$.

\subsection {Lifting Schubert exhaustions}

As in $\S\ref {incidence geometry}$ let $B$ be a Borel
subgroup of $G$ which contains an Iwasawa component $AN$,
and let $S$ be an $(n-q)$--dimensional $B$--Schubert variety
which has nonempty intersection with the base cycle $C_0$.
Recall that the intersection $S\cap D$ is contained in
$\mathcal O_S$ and its components $\Sigma $ are orbits $A_0N_0(z)$
of the intersection points $z\in S\cap C_0$.
\medskip

We fix a Schubert slice $\Sigma $.  If $C \in \mathcal M_D$ then
$C\cap\Sigma$ consists of a single point.  Thus $\Sigma$ defines
a fibration $\pi_{_\Sigma}: \mathcal M_D \to \Sigma $ by 
$C\mapsto C\cap \Sigma$.  The relevant family of 
Schubert varieties is defined by allowing $B$ to vary
under the condition that it contains $AN$.  In fact
this is just the closed $G_0$--orbit in $G/B$, in other words the orbit 
on which $K_0$ acts transitively.  Hence the family of Schubert
varieties determined by some initial choice $S_0$ is
just $\{k(S_0)\mid k\in K_0\}$.  Thus we are interested
in the family $\{k(\Sigma )\mid k\in K_0\}$ of Schubert slices.
\medskip

If the exhaustion $r_{_S}$ of $S$ is defined as above
by the section $s$, then the exhaustion $r_{_{k(S)}}$ of $k(S)$
is defined by $k(s)$. Hence 
\begin {equation}\label {translation}
r_{_{k(S)}}\circ \pi_{_{k(\Sigma )}}(C)=r_{_S}\circ \pi _{_\Sigma} (k(C))\,.
\end {equation}
Provisionally define $r_{_{\mathcal M_D}}$ to be the
supremum of the plurisubharmonic functions
$r_{_{k(S)}}\circ \pi_{_{k(\Sigma )}}$.  Since $\Sigma $ is just one
connected component of $S\cap D$, the final definition of
$r_{_{\mathcal M_D}}$ is given by maximizing over all such components.

\begin {proposition}
The function $r_{_{\mathcal M_D}}: \mathcal M_D \to \mathbb R$ is
a $K_0$--invariant continuous plurisubharmonic exhaustion.
\end {proposition}

\begin {proof}
It is a general fact that the supremum of a family of 
plurisubharmonic functions is plurisubharmonic.  It follows
from $(\ref {translation})$ that the provisionally defined
function arises by maximizing a fixed continuous function 
over a compact group action. Hence that function is continuous
and, as the maximum of finitely many such functions, it
then follows that $r_{_{\mathcal M_D}}$ is continuous.  Since
we have taken the supremum over a $K_0$--invariant family
of functions, $r_{_{\mathcal M_D}}$ is also $K_0$--invariant.
\medskip

Finally, every boundary point of $\mathcal M_D$ in the  
affine homogeneous space $G/K$ is contained in some
$K_0$--translate $k(H_0)$ of the incidence hypersurface
$H_0$ defined by $B_{S_0}$ (\cite {FHW}; see Section 2 above).
This implies that if $\{C_n\}$ is a sequence in $\mathcal M_D$
which converges to such a boundary point, then
$r_{_{\mathcal M_D}}(C_n) \to  \infty $.  If the sequence
$C_n$ diverges in $G/K$, then we may assume that it
converges to a point in the wonderful compactification of $G/K$
which is also in such a translate $k(H_0)$ (\cite {HH}).  
Thus $r_{_{\mathcal M_D}}(C_n)\to \infty $ for any sequence 
$\{C_n\}$ which diverges in $\mathcal M_D$ and therefore
$r_{_{\mathcal M_D}}$ is an exhaustion.
\end {proof} 

\subsection {Transfer to the flag domain}

The map $\nu :\mathfrak X_D\to \mathcal M_D$ is proper;
in fact, it is a trivial fiber bundle with fiber $C_0$.
Thus the lift $r_{_{\mathfrak X_D}} := r_{_{\mathcal M_D}}\circ \nu $
is a proper exhaustion.  Transferring to the flag domain,
we define $r_{_D}: D \to \mathbb R$ by 
$$
r_{_D(y)} := \mathrm {inf}\{r_{_{\mathfrak X_D}}(x)\mid x \in \mu ^{-1}(y)\}\,.
$$

The following follows from the fact that $r_{_{\mathfrak X_D}}$
is proper and $\mu :\mathfrak X_D\to D$ is an open map.

\begin {lemma}
The function $r_{_D}: D \to \mathbb R$ is a continuous $K_0$--invariant
exhaustion.
\end {lemma}

\begin {proof}
Since $r_{_{\mathfrak X_D}}$ is $K_0$--invariant and 
$\mu :\mathfrak X_D\to D$ is $K_0$--equivariant,
it is immediate that $r_{_D}$ is $K_0$--invariant.
\medskip

To prove that $r_{_D}$ is continuous, we let $\{y_n\} \to y$
in $D$.  Since $r_{_{\mathfrak X_D}}$ is a continuous exhaustion,
$r_{_D}(y_n)=r_{_{\mathfrak X_D}}(x_n)$ for some $x_n\in \mu ^{-1}(y_n)$
and we may assume that $\{x_n\} \to x$.  
This is due to the fact that if $\{x_n\}$ is unbounded in $\mathfrak X_D$,
then $\nu (x_n)$ is unbounded in $\mathcal M_D$ and as a result
$$
r_D(y_n)=r_{\mathfrak X_D}(x_n)=r_{\mathcal M_D}\circ \nu (x_n)\to \infty
$$
On the other hand, $\overline {\mathrm {lim}}r_D(y_n)$ is bounded by
$r_D(y)$.  For similar reasons, if  $r_{_D}(y)$ were not
$r_{_{\mathfrak X_D}}(x)$, then there would exist $\tilde x\in \mu^{-1}(y)$
with $r_{_D}(y)=r_{_{\mathfrak X_D}}(\tilde x) < r_{_{\mathfrak X_D}}(x)$.  
But $\mu $ is an open map and, contrary to assumption, we would be 
able to find $\tilde x_n\in \mu ^{-1}(y_n)$ with 
$r_{_{\mathfrak X_D}}(\tilde x_n) < r_{_{\mathfrak X_D}}(x_n)$.  Thus $r_{_D}$
is continuous. 
\medskip

To see that $r_D$ is an exhaustion, let $\{y_n\}$ be a
divergent sequence in $D$ and choose $\{x_n\}$ with
$x_n\in \mu^{-1}(y_n)$ and $r_D(y_n)=r_{\mathfrak X_D}(x_n)$.
It follows that $\{x_n\}$ is divergent as well, and since
$r_{\mathfrak X_D}$ is an exhaustion, $r_D(y_n)\to \infty$.
\end {proof}

\section {Pseudoconvexity}\label{sec5}

Since $r_{_D}$ is only known to be continuous, it doesn't make
sense to discuss its Levi form at a point of the boundary
of $D_r:= \{r_{_D} < r\}$. In the case of continuous plurisubharmonic 
exhaustions this causes no difficulties. However, as
is pointed out in \cite {ES}, subtle difficulties arise
in the $q$--convex case. 
\medskip

Let us say that a continuous exhaustion $h:D\to \mathbb R$
is $q$--pseudoconvex if every $p\in D$ has an open neighborhood
$U$ equipped with a smooth function $\tilde h$ with
$\tilde h(p)=h(p)$, $\tilde h\leqq h\vert U$, and the Levi form
$\mathcal L(\tilde h)=\frac{i}{2}\partial \bar \partial \tilde h$
has at least $n-q$ positive eigenvalues at every point of $U$.
\medskip

In particular the region $\{\tilde h<\tilde h(p)\}$ contains 
$D\cap U$ and has a
common boundary point with it at $p$, and near $p$ it has smooth
q--convex boundary.  In this sense our definition of a continuous
q--pseudoconvex function can be regarded as being natural.
For the purposes of smoothing, so that for example the theorems of
Andreotti and Grauert can be directly applied, it may be necessary to
require a somewhat stronger property.  It is quite possible that our
exhaustion is indeed more regular  than we have shown, but proving 
this in detail will require further study.

\begin {proposition}
The exhaustion $r_{_D}: D \to \mathbb R$ is $q$--pseudoconvex.
\end {proposition}

\begin {proof}
Let $y\in \{r_{_D} = r\}$. Let $x\in \mathfrak X_D$ be such that
$r_{_D}(y)=r_{\mathfrak X_D}(x)$ and recall that $x=(y,C)$ with
$y\in C$. Now by definition
$r_{_{\mathfrak X_D}}(x)=r_{_{\mathcal M_D}}(C)$, and $r_{_{\mathcal M_D}}(C)$ 
can be computed as $r_{_S} \circ \pi_{_\Sigma} (C)$,
for some Schubert variety $S$ in the defining family and
some slice $\Sigma $. Since $r_{_{\mathcal M_D}}\geqq r_{_S}\circ \pi_{_\Sigma} $,
it follows that
$\{r_{_{\mathcal M_D}}\leqq r\}\subset \{r_{_S}\circ \pi _{_\Sigma} \leqq r\}$.
\medskip

Now lift $r_{_S} \circ \pi _{_\Sigma} $ and $r_{_D}$ to functions
$h_S$ and $h_D$ on $\mathfrak X_D$.  We consider these
functions in a small neighborhood $W$ of $x=(y,C)$.
Note that they agree on the $\mu $--fiber 
$\mathcal F:=\mu ^{-1}(y)$.  
Furthermore,
$h_D\geqq h_S$ on a neighborhood of $(y,C)$ and thus we may
assume that this holds on $W$.
\medskip

By shrinking $W$ if necessary, we may choose a
closed complex submanifold $T$ of $W$ which is transversal
to $\mathcal F$ and is mapped biholomorphically onto
its open image $U$ in $D$. Define $r_{_T}$ on $U$ by restricting
$h_S$ to $T$ and using the identification $T=U$. Finally,
choose $T$ in such a way that it contains a closed 
$(n-q)$--dimensional complex submanifold $\Delta $ through $x$
so that the restriction of $\pi_{_\Sigma} \circ \nu $ to
$\Delta $ is biholomorphic onto an open subset of $\Sigma $.
These choices can be made, because $\mathcal F$ is both
the $\mu $--fiber over $y$ and the $\pi _{_\Sigma} $--fiber through $C$.
By construction $\{r_{_T}\leqq r\}$ contains $D_r\cap U$.
Furthermore, its restriction to the $\mu $-image of 
$\Delta $ can be regarded as the restriction of $r_{_S}$ to
the $\pi _{_\Sigma} \circ \nu $--image of $\Delta $.  Since
$r_{_S}$ is strictly plurisubharmonic, it follows that
$\mathcal L(r_T)$ has the desired number of positive
eigenvalues.
\end {proof}


\small
\begin {thebibliography} {X}

\bibitem{AnG}
A. Andreotti \& H. Grauert, Th\' eor\` emes de finitude pour la cohomologie
des espaces complexes,  Bull. Soc. Math France {\bf 90} (1962), 193--259.

\bibitem{B}
D. Barlet,
Convexit\'e de l'espace des cycles, Bull. Soc. Math. de France {\bf 106}
(1978), 373--397.

\bibitem{ES}
M. G. Eastwood \& G. V. Souria,
Cohomologically complete and pseudoconvex domains,
Comment. Math. Helv. {\bf 55} (1980), 413--426.

\bibitem{FHW}
G. Fels, A. T. Huckleberry \& J. A. Wolf,
``Cycle Spaces of Flag Domains:
A Complex Geometric Viewpoint.'' 
Progress in Mathematics, {\bf 245}. Birkh\" auser Boston, 
Boston, 2006.

\bibitem{HH}
J. Hong \& A. T. Huckleberry,
On closures of cycle spaces of flag domains,
Manuscripta Math. {\bf 121} (2006), 317--327. 

\bibitem{KS1}
B. Kr\"otz \& R. Stanton,
Holomorphic extension of representations, I:
Automorphic functions, Ann. of Math. {\bf 159} (2004), 641--724. 

\bibitem{KS2}
B. Kr\"otz \& R. Stanton,
Holomorphic extension of representations, II: geometry and harmonic analysis,
Geom. Funct. Anal. {\bf 15} (2005), 190--245.

\bibitem{SW}
W. Schmid \& J. A. Wolf, A vanishing theorem for open orbits on complex flag
manifolds, Proc. Amer. Math. Soc. {\bf 92} (1984), 461--464.

\bibitem{WeW}
R. O. Wells, Jr., \& J. A. Wolf,
Poincar\' e series and automorphic cohomology on flag domains.
Annals of Math. {\bf 105} (1977), 397--448.

\bibitem{W2}
J. A. Wolf,
The action of a real semisimple Lie group on a complex
manifold, {\rm I}: Orbit structure and holomorphic arc components,
Bull. Amer. Math. Soc. {\bf 75} (1969), 1121--1237.

\bibitem{W8}
J. A. Wolf,
The Stein condition for cycle spaces of open orbits on complex
flag manifolds,
Annals of Math. {\bf 136} (1992), 541--555.

\bibitem{W9}
J. A. Wolf,
Exhaustion functions and cohomology vanishing theorems for open orbits
on complex flag manifolds,
Math. Research Letters {\bf 2} (1995), 179--191.

\end {thebibliography}
\end {document}